\documentclass[12pt]{amsart}   % reqno number on RHS of equations.
\usepackage{amscd,amssymb,cite,amsmath,latexsym,comment,epic,eepic,euscript}
\usepackage{graphicx}
\usepackage{enumerate}
\usepackage[initials]{amsrefs}
\usepackage[all]{xy}
\usepackage{tikz}
\usetikzlibrary{decorations.markings,arrows, decorations.pathreplacing}

\theoremstyle{plain}
\newtheorem{thm}{Theorem}[section]
\newtheorem{cor}[thm]{Corollary} %%Delete [thm] to re-start numbering
\newtheorem{lemma}[thm]{Lemma} %%Delete [thm] to re-start numbering
\newtheorem{prop}[thm]{Proposition}

\theoremstyle{remark}

\theoremstyle{definition}
\newtheorem{defi}[thm]{Definition}

\newtheorem{remark}[thm]{Remark}

\usepackage{amscd,amssymb,comment,epic,eepic,euscript}
\usepackage{graphicx}
\usepackage{enumerate}
\usepackage[initials]{amsrefs}
\usepackage[all]{xy}
\newcount\theTime
\newcount\theHour
\newcount\theMinute
\newcount\theMinuteTens
\newcount\theScratch
\theTime=\number\time
\theHour=\theTime
\divide\theHour by 60
\theScratch=\theHour
\multiply\theScratch by 60
\theMinute=\theTime
\advance\theMinute by -\theScratch
\theMinuteTens=\theMinute
\divide\theMinuteTens by 10
\theScratch=\theMinuteTens
\multiply\theScratch by 10
\advance\theMinute by -\theScratch

\def\today{{\number\day\space
 \ifcase\month\or
  January\or February\or March\or April\or May\or June\or
  July\or August\or September\or October\or November\or December\fi
 \space\number\year}}

\newcommand{\beq}{\begin{equation}}
\newcommand{\eeq}{\end{equation}}
\newcommand{\bre}{\begin{remark}}
\newcommand{\ere}{\end{remark}}
\newcommand{\beqno}[1]{\begin{equation}\label{#1}}

\newcommand{\mref}[1]{(\ref{#1})}
\newcommand\bZ{\mathbb Z}
\DeclareMathOperator{\Hom}{Hom}

\begin{document}

\title[Additive functions on co-compact coverings]{A short note on additive functions on Riemannian co-compact coverings}

\author[Kha]{Minh Kha}
\address{M.K., Department of Mathematics, Texas A\&M University,
College Station, TX 77843-3368, USA}
\email{kha@math.tamu.edu}
%\date{\timeanddate}

\begin{abstract}
The main purpose of this note is to provide a topological approach to defining additive functions on Riemannian  co-compact  normal coverings.
\end{abstract}
\maketitle
%%%%%%%%%%%%%%%%%%%%%%%%%%
\section{Introduction}
Let $X$ be a connected smooth Riemannian manifold of dimension $n$ equipped with an isometric, properly discontinuous, free, and co-compact action of a discrete deck group $G$. Notice that the deck group $G$ is finitely generated due to the \v{S}varc-Milnor lemma and hence, $\Hom(G, \mathbb{R})$ is finite dimensional. Furthermore, the orbit space $M:=X/G$ is a compact Riemannian manifold when equipped with the metric pushed down from $X$. The action of an element $g \in G$ on $x \in X$ is denoted by $g\cdot x$. Let $\pi$ be the covering map from $X$ onto $M$. Thus, $\pi(g\cdot x)=\pi(x)$ for any $(g,x) \in G \times X$.

Following \cite{LinPinchover}, we define the class of additive functions on the covering $X$ as follows:
\begin{defi}
\label{add}
$ $ 

\begin{itemize}
\item
A real smooth function $u$ on $X$ is said to be \textbf{additive} if there is a homomorphism $\alpha: G \rightarrow \mathbb{R}$ such that 
\begin{equation*}
u(g\cdot x)=u(x)+\alpha(g), \quad \mbox{for all} \quad (g,x) \in G \times X.
\end{equation*}
We denote by $\mathcal{A}(X)$ the space of all additive functions on $X$.
\item

A map $h$ from $X$ to $\mathbb{R}^m$ ($m \in \mathbb{N}$) is called a vector-valued additive function on $X$ if every component of $h$ belongs to $\mathcal{A}(X)$.
\end{itemize}
\end{defi}
We remark that additive functions on co-compact covers appeared in various results such as studying the structure of positive $G$-multiplicative type solutions \cite{Ag2, LinPinchover}, describing the off-diagonal long time asymptotics of the heat kernel \cite{KoSu} and the Green's function asymptotics of periodic elliptic operators \cite{Kha} on a noncompact abelian cover of a compact Riemannian manifold.

A direct construction of additive functions on $X$ can be found in either \cite[Section 3]{KP2} or \cite[Remark 2.6]{LinPinchover}. However, this construction depends on the choice of a fundamental domain for the base $M$ in $X$. A more invariant approach to defining additive functions on covers was mentioned briefly in \cite{Ag2, KP2}. Our goal in this note is to present the full details of this approach for any co-compact covering.
%%%%%%%%%%%%%%%%%%%%%%%%%%%%%%
\section{Additive functions on co-compact normal coverings}
We begin with the following notion (see \cite{Ag2, KP2}):
\begin{defi}
\label{1form}
Let $H_{DR}^1(M), H_{DR}^1(X)$ be De Rham cohomologies of $M$ and $X$, respectively. We denote by $\Omega^1(M; G)$ the image in $H_{DR}^1(M)$ of the set of all closed differential 1-forms $\omega$ on $M$ (modulo the exact ones) such that their lifts $\omega$ to $X$ are exact.
In other words, $\Omega^1(M; G)$ is the kernel of the homomorphism
$$\pi^{*}: H_{DR}^1(M) \rightarrow H_{DR}^1(X),$$
where $\pi^{*}$ is the induced homomorphism of the covering map $\pi: X \rightarrow M$.
\end{defi}
By De Rham's theorem, $\Omega^1(M; G)$ is a finite dimensional vector space. Indeed, more is true:
\begin{lemma}
\label{DeRham}
$\Omega^1(M, G) \cong \Hom(G, \mathbb{R})$.
\end{lemma}
\begin{proof}
By Hurewicz's theorem (see e.g., \cite{Lee}), the homologies $H_1(M)$ and $H_1(X)$ are isomorphic to the abelianizations of the fundamental groups $\pi_1(M)$ and $\pi_1(X)$, respectively. Therefore, we can identify De Rham cohomologies $H_{DR}^1(M)$ and $H_{DR}^1(X)$ with $\Hom(\pi_1(M), \mathbb{R})$ and $\Hom(\pi_1(X), \mathbb{R})$, correspondingly. Since $X$ is a normal covering of $M$, $\pi_1(X)$ is a normal subgroup of $\pi_1(M)$ and moreover, the sequence
$$0 \rightarrow \pi_1(X) \rightarrow \pi_1(M) \rightarrow G \rightarrow 0$$
is exact. 
Because $\Hom(\cdot, \mathbb{R})$ is a contravariant exact functor, we deduce the exactness of the following sequence of vector spaces:
$$0 \rightarrow \Hom(G, \mathbb{R})  \rightarrow H_{DR}^1(M) \rightarrow  H_{DR}^1(X) \rightarrow 0.$$
Hence, $\Omega^1(M, G)$ is isomorphic to $\Hom(G, \mathbb{R})$.
\end{proof}

Fixing a base point $x_0 \in X$. For any closed 1-form $\omega$ on $M$ such that its lift to $X$ is exact,
there exists a unique function $f_{\omega} \in C^{\infty}(X, \mathbb{R})$ such that $\pi^* \omega=df_{\omega}$ and $f_{\omega}(x_0)=0$. Equivalently,
$$f_{\omega}(x)=\int_{x_0}^x \pi^*\omega, \quad \forall x \in X.$$
\begin{lemma}
\label{additive_form}
For such 1-form $\omega$, we have:
\begin{enumerate}[i)]
\item Fix any $g \in G$, then $f_{\omega}(g \cdot x)-f_{\omega}(x)$ is independent of $x \in X$.

\item If $\pi^*\omega=0$ then $\omega=0$.
\end{enumerate}
\end{lemma}

\begin{proof}
$ $

\begin{enumerate}[i)]
\item For each $g \in G$, let $L_{g}$ be the diffeomorphism of $X$ that maps $x$ to $g \cdot x$. Since $\pi \circ L_{g}=\pi$, we get $df_{\omega}=\pi^*\omega=L_{g}^*\pi^*\omega=L_{g}^*df_{\omega}=dL_{g}^*f_{\omega}=d(f_{\omega}\circ L_{g})$. Thus, $d(f_{\omega}-f_{\omega}\circ L_{g})=0$ and so, $f_{\omega}\circ L_{g}-f_{\omega}$ is constant since  $X$ is connected.

\item 
Fix any point $p \in M$. We pick an evenly covered open subset $U$ of $M$ such that it contains $p$. Then there is a smooth local section $\sigma: U \rightarrow X$, i.e., $\pi \circ \sigma=id_{|U}$ (see e.g., \cite[Proposition 4.36]{Lee}). Hence, $\omega(p)=\sigma^{*}\pi^{*}\omega(p)=0$. 

%Suppose $(U,x)$ is a coordinate chart by reducing $U$ if necessary. On $U$, assume that 
%$$\omega_{|U}=\sum_{i=1}^n g_i dx_i,$$
%for some function $g_i \in C^{\infty}(M, \mathbb{R})$, $i=1,\dots,n$.
%Now let $\tilde{U}$ be any component of $\pi^{-1}(U)$ so that $\pi_{|\tilde{U}}$ is a diffeomorphism from $\tilde{U}$ onto $U$. We have 
%$$0=\pi^*\omega_{|\tilde{U}}=\sum_{i=1}^n (g_i\circ \pi) d(x_i \circ \pi).$$
%Since $(\tilde{U}, x \circ \pi)$ is a coordinate chart, it follows that $g_i \circ \pi=0$ on $\tilde{U}$. Hence, $g_i=0$ on $U$ for any $i=1,\dots,n$, which proves $\omega_{|U}=0$ and thus $\omega(p)=0$ for any $p \in M$. 
\end{enumerate}
\end{proof}

On $\mathcal{A}(X)$, we introduce an equivalent relation $\sim$ as follows: $f_1 \sim f_2$ in $\mathcal{A}(X)$ if and only if $f_1-f_2=f \circ \pi$ for some function $f\in C^{\infty}(M,\mathbb{R})$. 

By Lemma \ref{additive_form} (i), the map $\omega \mapsto f_{\omega}$ induces the following linear map
\begin{equation}
\label{Lambda}
\begin{split}
\Lambda: \Omega^1(M, G)&\rightarrow \mathcal{A}(X)/\sim
\\ [\omega] &\mapsto [f_{\omega}],
\end{split}
\end{equation}
where $[\omega]$, $[f_{\omega}]$ are the equivalent classes of $\omega$, $f_{\omega}$ in $\Omega^1(M, G)$ and $\mathcal{A}(X)/\sim$, correspondingly. We now claim that $[\omega]=0$ if and only if $[f_{\omega}]=0$, and hence $\Lambda$ is an injective linear map. Indeed, due to Lemma \ref{additive_form} (ii), the condition that $\omega$ is exact is equivalent to $\pi^*\omega=d(f \circ \pi)$ for some $f \in C^{\infty}(M,\mathbb{R})$. But this is the same as $df_{\omega}=d(f\circ \pi)$, or $[f_{\omega}]=0$.

Consider an additive function $f$ on $X$. According to the definition, there exists a unique group homomorphism $\ell_f: G \rightarrow \mathbb{R}$ such that $f(g \cdot x)=f(x)+\ell_f(g)$ for any $g \in G$, $x \in X$. Then the map $f \mapsto \ell_f$ induces the linear map
\begin{equation}
\label{Upsilon}
\begin{split}
\Upsilon: \mathcal{A}(X)/\sim &\rightarrow \Hom{(G, \mathbb{R})}
\\ [f] &\mapsto \ell_f,
\end{split}
\end{equation}
which is injective.

Then the composition $\Upsilon \circ \Lambda$ is also injective. By Lemma \ref{DeRham}, $\dim_{\mathbb{R}}\Omega^1(M,G)=\dim_{\mathbb{R}}\Hom{(G, \mathbb{R})}<\infty$. These facts together imply that the linear maps $\Upsilon$ and  $\Lambda$ are isomorphism. We conclude:

\begin{thm}
\label{correspondence}
The three vector spaces $\Omega^1(M, G)$, $\mathcal{A}(X)/\sim$ and $\Hom{(G, \mathbb{R})}$ are isomorphic to each other. 
\end{thm}
In particular, we obtain:
\begin{cor}
Assume that $G=\bZ^d$. Then there is a smooth $\mathbb{R}^d$-valued function $h$ on $X$ such that for any $(g,x) \in \bZ^d \times X$, 
\begin{equation}
\label{additivity}
h(g\cdot x)=h(x)+g.
\end{equation}
\end{cor}
The following standard proposition says that given any additive function $u$ on $X$, then one can pick a harmonic additive function $f$ such that $f-u$ is $G$-periodic.
\begin{prop}
\label{harmonic}
For any $\ell$ in $\Hom(G, \mathbb{R})$, there exists a unique (modulo a real constant) \textbf{harmonic} function $f$ on $X$ such that for any $(g, x) \in G \times X$, we have
\begin{equation}
\label{harmonic_additive}
f(g \cdot x)=f(x)+\ell(g).
\end{equation} 
\end{prop}
\begin{proof}
First, we show the existence part. Due to Theorem \ref{correspondence}, let $\tilde{f}$ be a function on $X$ satisfying $\tilde{f}(g \cdot x)=\tilde{f}(x)+\ell(g)$ for any $(g, x) \in G \times X$. 
We recall the isomorphism $\Lambda$ defined in \mref{Lambda}. We put $\alpha:=\Lambda^{-1}([\tilde{f}])\in \Omega^1(M, G)$. By the Hodge theorem, there exists a unique harmonic 1-form $\omega$ on $M$ such that $[\omega]=\alpha$ in $H^1_{DR}(M)$. Let $f$ be a smooth function such that $f \in \mathcal{A}(X)$ and $\pi^*\omega=df$. 
Then $f$ satisfies \mref{harmonic_additive} since $[f]=[\tilde{f}]$ in $\mathcal{A}(X)/\sim$. Thus, it is sufficient to show that $f$ is harmonic on $X$. We denote by $\delta_X$, $\Delta_X$ and $\delta_M$, $\Delta_M$ the codifferential and Laplace-Beltrami operators on $X$ and $M$, respectively. Since the covering map $\pi$ is a local isometry between $X$ and $M$, its pullback $\pi^{*}$ intertwines the codifferential operators, i.e.,
$$\delta_X \pi^*=\pi^* \delta_M.$$
Since $\Delta_M\omega=0$, it follows that $\delta_M \omega=0$. Thus,
\begin{equation*}
\Delta_X f=\delta_X df=\delta_X \pi^*\omega=\pi^* \delta_M\omega=0.
\end{equation*}

For the uniqueness part, let $f_1$ and  $f_2$ be any two harmonic functions on $X$ such that \mref{harmonic_additive} holds for each of these functions. Since $f_1-f_2$ is $G$-periodic, it can be pushed down to a real function $f$ on $M$. Moreover, $\pi^*\Delta_M f=\Delta_X \pi^*f=\Delta_X(f_1-f_2)=0$. Therefore, $f$ must be constant since it is a harmonic function on a compact, connected Riemannian manifold $M$. Thus, $f_1-f_2$ is constant. 
\end{proof}

\begin{cor}
Fixing a base point $x_0$ in $X$. Then to each $\alpha \in \Hom(G, \mathbb{R})$, there exists a unique harmonic function $f_{\alpha}$ defined on $X$ such that $f_{\alpha}(x_0)=0$ and $\Upsilon([f_{\alpha}])=\alpha$, where $\Upsilon$ is introduced in \mref{Upsilon}. Consequently, 
\begin{equation*}
\mathcal{A}(X)=\bigsqcup_{\alpha \in \Hom(G, \mathbb{R})}\big\{f_{\alpha}+\varphi \hspace{2pt} \mid \hspace{2pt} \varphi \hspace{3pt} \mbox{is periodic}\hspace{1pt}\big\}.
\end{equation*}
\end{cor}

\begin{remark}
When $G=\bZ^d$, 
%the Albanese map introduced in \cite[Section 2]{KoSu} is a harmonic vector-valued additive function satisfying \mref{additivity}. Consequently, 
the Albanese pseudo-metric $d_G$ introduced in \cite[Section 2]{KoSu} is actually the pseudo-distance arising from \textbf{any} harmonic vector-valued additive function $h$ satisfying \mref{additivity}, i.e., $d_G(x,y)=|h(x)-h(y)|$ for any $x,y \in X$.
\end{remark}
%In particular, the asymptotics \mref{main_asymp} and \mref{main_asymp_KR} can be described in terms of the Albanese map $\tilde{\Phi}_{\Gamma}$ and the Albanese pseudo-distance $d_{\Gamma}$.
%%%%%%%%%%%%%%%%%
\section{Acknowledgements}
The author was partially supported by the NSF grant DMS-1517938.
%%%%%%%%%%%%%%%%%%

\end{document}